\documentclass[]{scrartcl}  
\usepackage{amsmath}
\usepackage{amsthm}                 
\usepackage{amssymb}
\usepackage{graphicx}
\usepackage{psfrag}
\usepackage{caption}
\usepackage{booktabs}
\usepackage{xypic}
\usepackage{mdwlist}

\theoremstyle{plain}
\newtheorem{theorem}{Theorem}

\newtheorem{lemma}[theorem]{Lemma}
\newtheorem{corollary}[theorem]{Corollary}
\newtheorem{proposition}[theorem]{Proposition}

\theoremstyle{definition}
\newtheorem{remark}[theorem]{Remark}

\newenvironment{example}
{\medskip\noindent{\bf Example}}{\smallskip\par}

\newcommand{\Example}[2]{\begin{example}{{ \textbf{#1}.} #2}\end{example}}

\newcommand{\CC}{{\mathbb{C}}}

\newcommand{\PP}{{\mathbb{P}}}

\newcommand{\ZZ}{{\mathbb{Z}}}

\newcommand{\calO}{{\cal O}}

\newcommand{\calT}{{\cal T}}

\newcommand{\DDD}{D}

\newcommand{\E}[2]{E^{{#1}}_{{#2}}}

\newcommand{\Hom}{\mathrm{Hom}}
\newcommand{\Ext}{\mathrm{Ext}}

\newcommand{\TTT}{\mathsf{T}\!}
\newcommand{\chern}{\mathrm{ch}}

\newcommand{\lla}{$\longleftarrow$}

\newcommand{\includegraph}[1]{ \includegraphics[width=0.7\textwidth,trim=0 0 2mm 2mm,clip]{#1} }

\title{A geometric construction of Coxeter-Dynkin diagrams of bimodal singularities}
\author{Wolfgang Ebeling and David Ploog
\date{}
\thanks{Supported by the DFG priority program SPP 1388 ``Representation Theory'' (Eb 102/6--1).\newline
AMS Math.\ Subject Classification (2010): 32S25, 18E30, 53D37.
}
}

\begin{document}

\maketitle

\begin{abstract} We consider the Berglund-H\"ubsch transpose of a bimodal invertible polynomial and construct a triangulated category associated to the compactification of a suitable deformation of the singularity. This is done in such a way that the corresponding Grothendieck group with the (negative) Euler form can be described by a graph which corresponds to the Coxeter-Dynkin diagram with respect to a distinguished basis of vanishing cycles of the bimodal singularity.
\end{abstract}

\section*{Introduction}
Let $f(x,y,z)$ be a weighted homogeneous polynomial which has an isolated singularity at the origin $0 \in \CC^3$. An important invariant of $f$ is a Coxeter-Dynkin diagram with respect to a distinguished basis of vanishing cycles in the Milnor fibre of $f$. It determines the monodromy of the singularity as the corresponding Coxeter element. The vanishing cycles can be chosen to be (graded) Lagrangian submanifolds of the Milnor fibre. A distinguished basis of such vanishing Lagrangian cycles can be categorified to an $A_\infty$-category ${\rm Fuk}^\to(f)$ called the directed Fukaya category of $f$. Its derived category $D^b{\rm Fuk}^\to(f)$ is, as a triangulated category, an invariant of the polynomial $f$.

On the other hand, one can consider the bounded derived category of coherent sheaves on a resolution of the singularity or of a compactification of the Milnor fibre as in \cite{EP}. The homological mirror symmetry conjecture states that there should be a relation between these categories for mirror symmetric singularities.

In \cite{ET}, the first author and A.~Takahashi considered a mirror symmetry in a specific class of weighted homogeneous polynomials in three variables, namely the so called invertible polynomials. The mirror symmetry is given by the Berglund--H\"ubsch transpose $f^T$ of $f$. They generalised Arnold's strange duality for the 14 exceptional unimodal singularities to this wider class. They defined Dolgachev and Gabrielov numbers for such invertible polynomials and showed that the Dolgachev numbers of $f$ coincide with the Gabrielov numbers of $f^T$ and the Gabrielov numbers of $f$ coincide with the Dolgachev numbers of $f^T$.

In the case of the 14 exceptional unimodal singularities, the Gabrielov numbers are directly related with a Coxeter-Dynkin diagram of the singularity. In \cite{EP}, it was shown that one can find a Coxeter-Dynkin diagram of the dual singularity in the bounded derived category of coherent sheaves on a resolution of the compactification of the Milnor fibre of $f$.

In this paper, we consider the bimodal singularities. They were also classified by V.~I.~Arnold. They fall into 8 infinite series starting with 6 classes where, setting one modulus equal to 0, one obtains weighted homogeneous polynomials. Besides these series, there are again 14 exceptional singularities. In these 6+14 classes one finds invertible polynomials. Coxeter-Dynkin diagrams for the bimodal singularities were computed in \cite{E83}. In this paper, we shall show that these Coxeter-Dynkin diagrams can be constructed geometrically in a way similar to \cite{EP} using suitable invertible polynomials and their Berglund-H\"ubsch transposes.

We would like to thank K.~Ueda for pointing out to us that some of the functions $F$ in Table~\ref{TabWPS} in the published paper were not quasismooth. This concerns the cases $Z_{18}$, $Q_{16}$, $S_{1,0}$ and $S_{16}$. The corresponding entries of Table~\ref{TabWPS} have now been corrected. This does not affect the results of the paper as it is shown for Example $S_{16}$.
\section{Invertible polynomials}
Let $f(x_1, \ldots , x_n)$ be a weighted homogeneous complex polynomial. This means that there are positive integers $w_1,\dots ,w_n$ and $d$ such that
\[ f(\lambda^{w_1} x_1, \dots, \lambda^{w_n} x_n) = \lambda^d f(x_1,\dots ,x_n) \]
for $\lambda \in \CC^\ast$. We call $(w_1,\dots ,w_n;d)$ a system of {\em weights}.
The weight system is said to be \emph{reduced} if $\text{gcd}(w_1,\dots ,w_n,d)=1$; otherwise it is
called \emph{non-reduced}.
Recall that a quasihomogeneous polynomial $f(x_1, \ldots , x_n)$ in $n$ variables is called {\em invertible} if it is of the form
\[
 f(x_1, \ldots, x_n)=\sum\limits_{i=1}^n a_i \prod\limits_{j=1}^n x_j^{E_{ij}}
\]
for some coefficients $a_i\in\CC^\ast$ and for a matrix $E=(E_{ij})$ with non-negative integer entries and with $\det E\ne 0$. For simplicity we can assume $a_i=1$ for $i=1, \ldots, n$. (This can be achieved by a suitable rescaling of the variables.) An invertible quasihomogeneous polynomial $f$ is called {\em non-degenerate} if it has (at most) an isolated critical point at the origin in $\CC^n$. An invertible polynomial has a {\em canonical system of weights}: This is the system of weights $W_f=
(w_1,\dots ,w_n;d')$ given by the unique solution of the equation
\begin{equation*}
E
\begin{pmatrix}
w_1\\
\vdots\\
w_n
\end{pmatrix}
={\rm det}(E)
\begin{pmatrix}
1\\
\vdots\\
1
\end{pmatrix}
,\quad
d':={\rm det}(E).
\end{equation*}
We may and will assume that $w_1$, \dots , $w_n$ and $d'$ are positive integers.
This system of weights is in general non-reduced. Define
\[ c_f:= {\rm gcd}(w_1, \dots  ,w_n,d'). \]
Let
\[ (q_1, \ldots, q_n; d) := (w_1/c_f, \ldots , w_n/c_f; d'/c_f) \]
be the corresponding reduced weight system. We define the \emph{Berglund-H\"{u}bsch transpose}
$f^T(x_1,\dots ,x_n)$ of an invertible polynomial $f(x_1,\dots ,x_n)$ by
\[
f^T(x_1,\dots ,x_n):=\sum_{i=1}^na_i\prod_{j=1}^nx_j^{E_{ji}}.
\]

\section{Weighted homogeneous bimodal singularities} \label{Sect:bi}
The bimodal singularities have been classified by Arnold \cite{ARuss, AInv}. They are characterised by the fact that the exceptional divisor of the minimal resolution is a Kodaira degenerate elliptic curve of type ${\rm I}_p^\ast$, $p \geq 0$, ${\rm  IV}^\ast$, ${\rm III}^\ast$, or ${\rm II}^\ast$ with a different neighbourhood \cite{K75, EW85}. In the classes ${\rm I}_0^\ast$, ${\rm  IV}^\ast$, ${\rm III}^\ast$, and ${\rm II}^\ast$ one can find weighted homogeneous polynomials. The list of classes with the names given by Arnold and their deformations is given in Table~\ref{TabKodaira}. We also indicate the number $r$ of components of the exceptional divisor with a self-intersection number different from $-2$. 

The 6 singularities of Kodaira type ${\rm I}_0^\ast$ are referred to as \emph{quadrilateral singularities} since they correspond to certain quadrangles in the hyperbolic plane in the same way as the 14 exceptional unimodal singularities correspond to triangles in the hyperbolic plane \cite{Dolgachev75}. The remaining singularities of Kodaira types ${\rm  IV}^\ast$, ${\rm III}^\ast$, ${\rm II}^\ast$ are called \emph{exceptional}.
\begin{table}[h]
\centering
\begin{tabular}{c @{\hspace{3em}} *{7}{c} }
\toprule
$r$ & ${\rm I}_0^\ast$ &&${\rm  IV}^\ast$ && ${\rm III}^\ast$ && ${\rm II}^\ast$  \\
\midrule
1 & $J_{3,0}$ & \lla & $E_{18}$ & \lla & $E_{19}$ & \lla & $E_{20}$ \\
1 & $Z_{1,0}$ & \lla & $Z_{17}$ & \lla & $Z_{18}$ & \lla & $Z_{19}$ \\
1 & $Q_{2,0}$ & \lla & $Q_{16}$ & \lla & $Q_{17}$ & \lla & $Q_{18}$ \\
\midrule[0.1pt]
2 & $W_{1,0}$ & \lla & $W_{17}$ & \lla & $W_{18}$ \\
2 & $S_{1,0}$ & \lla & $S_{16}$ & \lla & $S_{17}$ \\
\midrule[0.1pt]
3 & $U_{1,0}$ & \lla & $U_{16}$ \\
\bottomrule
\end{tabular}
\caption{Weighted homogeneous bimodal singularities} \label{TabKodaira}
\end{table}

In fact, in each of these classes one can find non-degenerate invertible polynomials in three variables.  
In Table~\ref{TabBi} there are chosen invertible polynomials for these classes and in each case, the 
Berglund-H\"ubsch transpose is indicated. Since the Berglund-H\"ubsch transpose will be our main concern,
we shall denote it by $f$ and the invertible polynomial for the bimodal singularity by $f^T$. We also 
indicate the Dolgachev numbers $\alpha_1,\alpha_2, \alpha_3$ and Gabrielov numbers $\gamma_1$, $\gamma_2$,
$\gamma_3$ for $f$ as defined in \cite{ET}. They are the Gabrielov numbers and Dolgachev numbers of the 
polynomial $f^T$ respectively by \cite{ET}. Note that these numbers depend on the polynomial $f$ and, in
general,  they differ from the Dolgachev numbers of the singularity in \cite{Dolgachev75}.
\begin{table}[h]
\centering
\begin{tabular}{cccccc}
\toprule
Name & $\gamma_1, \gamma_2, \gamma_3$  & $f^T$   & $f$ & $\alpha_1,\alpha_2, \alpha_3$ & Dual \\
\cmidrule(r){1-3} \cmidrule(l){4-6}
$J_{3,0}$ & $2,4,6$ & $x^6y+y^3+z^2$ & $x^6+xy^3+z^2$ & $2,3,10$ & $Z_{13}$ \\
$Z_{1,0}$ & $2,4,8$ & $x^5y + xy^3 +z^2$ & $x^5y + xy^3 +z^2$ & $2,4,8$ & $Z_{1,0}$ \\
$Q_{2,0}$ & $2,4,10$ & $x^4y + y^3 + xz^2$ & $x^4z + xy^3 +z^2$ & $3,3,7$ & $Z_{17}$\\
$W_{1,0}$ & $2,6,6$ & $x^6+y^2+yz^2$  & $x^6+y^2z+z^2$ & $2,6,6$ & $W_{1,0}$ \\
$S_{1,0}$ & $2,6,8$ & $x^5+xy^2+yz^2$ & $x^5y+y^2z+z^2$ & $3,5,5$ & $W_{17}$ \\
$U_{1,0}$ & $3,4,6$ & $x^3+xy^2+yz^3$  & $x^3y+y^2z+z^3$ & $3,4,6$ & $U_{1,0}$ \\
\cmidrule[0.1pt](r){1-3} \cmidrule[0.1pt](l){4-6}
$E_{18}$ & $3, 3, 5$ & $x^5z + y^3 + z^2$  & $x^5 + y^3 + xz^2$ & $2,3,12$ & $Q_{12}$ \\
$E_{19}$ & $2, 4, 7$ & $x^7y + y^3+z^2$  & $x^7 + xy^3+z^2$  & $2,3,12$ & $Z_{1,0}$  \\
$E_{20}$ & $2, 3, 11$ & $x^{11} +y^3 +z^2$ & $x^{11} +y^3 +z^2$ & $2, 3, 11$ & $E_{20}$ \\
\cmidrule[0.1pt](r){1-3} \cmidrule[0.1pt](l){4-6}
$Z_{17}$ & $3, 3, 7$ & $x^4z + xy^3 + z^2$ & $x^4y + y^3 + xz^2$ & $2,4,10$ & $Q_{2,0}$  \\
$Z_{18}$ & $2, 4, 10$ & $x^6y + xy^3 + z^2$ & $x^6y + xy^3 + z^2$ & $2, 4, 10$ & $Z_{18}$  \\
$Z_{19}$ & $2, 3, 16$ & $x^9 + xy^3 + z^2$  & $x^9y + y^3 + z^2$ & $2, 4, 9$ & $E_{25}$\\
\cmidrule[0.1pt](r){1-3} \cmidrule[0.1pt](l){4-6}
$Q_{16}$ & $3, 3, 9$ & $x^4z + y^3 + xz^2$ & $x^4z + y^3 + xz^2$ & $3, 3, 9$ & $Q_{16}$ \\
$Q_{17}$ & $2, 4, 13$ & $x^5y+y^3+xz^2$ & $x^5z+xy^3+z^2$  & $3,3,9$  & $Z_{2,0}$  \\
$Q_{18}$ & $2, 3, 21$ & $x^8+y^3+xz^2$ & $x^8z+y^3+z^2$ & $3, 3, 8$ & $E_{30}$  \\
\cmidrule[0.1pt](r){1-3} \cmidrule[0.1pt](l){4-6}
$W_{17}$ & $3, 5, 5$ & $x^5z+yz^2+y^2$ & $x^5+xz^2+y^2z$ & $2,6,8$ & $S_{1,0}$  \\
$W_{18}$ & $2, 7, 7$ & $x^7+ y^2+ yz^2$ & $x^7+ y^2z+ z^2$ & $2, 7, 7$ & $W_{18}$  \\
\cmidrule[0.1pt](r){1-3} \cmidrule[0.1pt](l){4-6}
$S_{16}$ & $3, 5, 7$ & $x^4y+xz^2+y^2z$ & $x^4y+xz^2+y^2z$ & $3, 5, 7$ & $S_{16}$  \\
$S_{17}$ & $2, 7, 10$ & $x^6+xy^2+yz^2$ & $x^6y+y^2z+z^2$ & $3,6,6$ & $X_{2,0}$  \\
\cmidrule[0.1pt](r){1-3} \cmidrule[0.1pt](l){4-6}
$U_{16}$ & $5, 5, 5$ & $x^5+y^2z+yz^2$ & $x^5+y^2z+yz^2$ & $5, 5, 5$ & $U_{16}$  \\
\bottomrule
\end{tabular}
\caption{Strange duality of the bimodal singularities} \label{TabBi}
\end{table}

In each case, the invertible polynomial $f$ defines another singularity whose name (in Arnold's 
notation) is also given in the table. Note that we have chosen two invertible polynomials in the 
singularity class $Z_{1,0}$ whose Berglund-H\"ubsch transposes lie in different classes of bimodal
singularities, namely $Z_{1,0}$ and $E_{19}$.

Coxeter-Dynkin diagrams with respect to distinguished bases of vanishing cycles for these 
singularities were determined in \cite{E83}. By  a Coxeter-Dynkin diagram we mean the following graph.
Let $(L, \langle - , - \rangle)$ be an integral lattice, i.e.\  $L$ is a finitely generated free
$\ZZ$-module equipped with a symmetric bilinear form $\langle - , - \rangle$ with values in $\ZZ$.
An element $e \in L$ with $\langle e,e\rangle =-2$ is called a 
\emph{root}. Such an element $e$ defines a
reflection
\[ s_e(x) = x- \frac{2\langle x, e \rangle}{\langle e,e \rangle}
          = x + \langle x,e \rangle e \text{ for } x \in L. \]
Let $B=(e_1, \ldots , e_n)$ be a basis of $L$ consisting of roots.
The symmetric bilinear form $\langle - , - \rangle$ with respect to this ordered basis is encoded by 
a graph, the so called \emph{Coxeter-Dynkin diagram} corresponding to the basis $B$, in the following way:
The vertices correspond to the basis elements $e_i$ and two vertices $e_i$ and $e_j$ with $i \neq j$ are 
joined by $|\langle e_i, e_j \rangle|$ edges which are dashed if $\langle e_i, e_j \rangle < 0$. 
The {\em Coxeter element} $\tau$ corresponding to $B$ is defined by
\[ \tau = s_{e_1} s_{e_2} \cdots s_{e_n}. \]

In the singularity case, we are interested in the Milnor lattice $L$ and a Coxeter-Dynkin diagram 
corresponding to a distinguished basis of vanishing cycles of the Milnor lattice. Then the Coxeter element corresponding to such a basis is the monodromy operator of the singularity.

According to \cite{E83}
(see also \cite{ET}), a Coxeter-Dynkin diagram with respect to a distinguished basis of vanishing cycles
of one of the bimodal singularities can be obtained by the following rule from the invariants of 
Table~\ref{TabCD}: Here $(\alpha_1, \alpha_2, \alpha_3)$ are the Dolgachev numbers of $f$.
The number $a$ is the \emph{Gorenstein parameter} of the canonical system of weights
$W_{f^T}=(w^T_1,w^T_2,w^T_3;d^T)$ of $f^T$, i.e.
\[ a:= d^T-w^T_1-w^T_2-w^T_3. \]
Let $T(\alpha_1,\alpha_2,\alpha_3)$ be the T-shaped graph of Figure~\ref{FigTpqr}.
\begin{figure}
\centering
\includegraphics[width=0.95\textwidth,trim=0 0 0 2mm,clip]{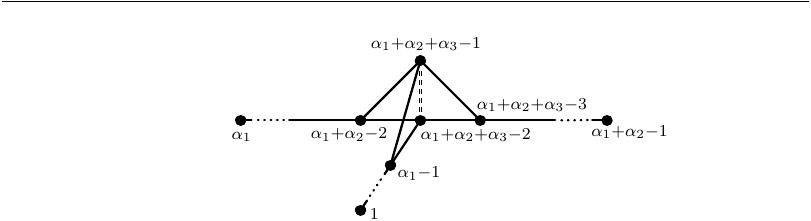}
\caption{The graph $T(\alpha_1, \alpha_2, \alpha_3)$} \label{FigTpqr}
\end{figure}

\begin{itemize}
\item If $a=2$ then  the diagram $T(\alpha_1,\alpha_2,\alpha_3)$ is extended by $\bullet_1 \mbox{---} \bullet_2$ where $\bullet_1$ is connected to the upper central vertex and $\bullet_2$ to the $\alpha_i - \beta_i-1$-th vertex from the outside of the $i$-th arm, unless $\beta_i=\alpha_i - 1$ ($i=1,2,3$).
\item If $a=3$ then  the diagram $T(\alpha_1,\alpha_2,\alpha_3)$ is extended by $\bullet_1 \mbox{---} \bullet_2 \mbox{---} \bullet_3$ where $\bullet_1$ is connected to the upper central vertex and $\bullet_3$ to the $\alpha_i - \beta_i-1$-th vertex from the outside of the $i$-th arm, unless $\beta_i=\alpha_i - 1$ ($i=1,2,3$).
\item If $a=5$ then  the diagram $T(\alpha_1,\alpha_2,\alpha_3)$
 is extended by $\bullet_1 \mbox{---} \bullet_2 \mbox{---} \bullet_3\mbox{---} \bullet_4\mbox{---} \bullet_5$
where $\bullet_1$ is connected to the upper central vertex and $\bullet_3$ to the $\alpha_i - \beta_i-1$-th vertex from the outside of the $i$-th arm, unless $\beta_i=\alpha_i - 1$ ($i=1,2,3$).
\end{itemize}
The numbering of the vertices of the complete graph is obtained by taking the new vertices as last vertices, in their indicated order.
\begin{table}[h]
\centering
\begin{tabular}{ccccc}
\toprule
Dual   & $c_f$   & $(\alpha_i,\beta_i), i=1,2,3$ & $a$ & Name \\
\midrule
$Z_{13}$   & 2 & $(2,1), (3,2), (10,7)$ & 2 & $J_{3,0}$ \\
$Z_{1,0}$ & 2  & $(2,1), (4,3), (8,5)$  & 2 & $Z_{1,0}$\\
$Z_{17}$  & 1 & $(3,2), (3,2), (7,4) $ & 2 & $Q_{2,0}$ \\
$W_{1,0}$ & 2 & $(2,1), (6,4), (6,4) $ & 2 & $W_{1,0}$ \\
$W_{17}$  & 1 & $(3,2), (5,3), (5,3) $ & 2 & $S_{1,0}$ \\
$U_{1,0}$ & 2 & $(3,1), (4,3), (6,4) $ & 2 & $U_{1,0}$ \\
\midrule[0.1pt]
$Q_{12}$ & 2 & $(2,1), (3,2), (12,8)$ & 2 & $E_{18}$\\
$Z_{1,0}$ & 3 & $(2,1), (3,2), (12,9)$ & 3 & $E_{19}$\\
$E_{20}$  & 1 & $(2,1), (3,2), (11,9)$ & 5 & $E_{20}$\\
\midrule[0.1pt]
$Q_{2,0}$ & 2 & $(2,1), (4,3), (10,6)$ & 2 & $Z_{17}$ \\
$Z_{18}$   & 1 & $(2,1), (4,3), (10,7)$ & 3 & $Z_{18}$\\
$E_{25}$   & 1 & $(2,1), (4,3), (9,7)$  & 5 & $Z_{19}$\\
\midrule[0.1pt]
$Q_{16}$  & 1 & $(3,2), (3,2), (9,5)$ & 2 & $Q_{16}$\\
$Z_{2,0}$  & 3 & $(3,2), (3,2), (9,6)$ & 3 & $Q_{17}$\\
$E_{30}$   & 1 & $(3,2), (3,2), (8,6)$ & 5 & $Q_{18}$\\
\midrule[0.1pt]
$S_{1,0}$  & 2 & $(2,1), (6,4), (8,5)$ & 2 & $W_{17}$\\
$W_{18}$  & 1 & $(2,1), (7,5), (7,5)$ & 3 & $W_{18}$\\
\midrule[0.1pt]
$S_{16}$   & 1 & $(3,2), (5,3), (7,4)$ & 2 & $S_{16}$\\
$X_{2,0}$  & 3 & $(3,2),(6,4), (6,4)$  & 3 & $S_{17}$\\
\midrule[0.1pt]
$U_{16}$   & 1 & $(5,3), (5,3), (5,3)$ & 2 & $U_{16}$\\
\bottomrule
\end{tabular}
\caption{Invariants of the singularities} \label{TabCD}
\end{table}

Note that in the cases where the canonical systems of weights of $f$ are reduced ($c_f=1$), the numbers $\beta_i$ of Table~\ref{TabCD} satisfy $a \beta_i \equiv  1 \, {\rm mod} \, \alpha_i$, $i=1,2,3$. Therefore, in these cases,  the invariants $(\alpha_1,\beta_1), (\alpha_2,\beta_2), (\alpha_3,\beta_3)$ are just the orbit invariants of the $\CC^\ast$-action on the corresponding singularity, by \cite{Dolgachev83}.

\section{Deformations and compactifications}
Our aim is to realize such a Coxeter-Dynkin diagram in a geometric way using the resolution of the compactification of a suitable deformation of the singularity $f(x,y,z)$ dual to the given singularity.

We consider one of the invertible polynomials $f(x,y,z)$ of Table~\ref{TabBi}. Let $(q_1,q_2,q_3;d)$ be the reduced weight system of $f$. We consider a suitable deformation $f_w$ of $f$ and a compactification of the level set $f_w=0$ in a weighted projective 3-space. Let
\[ q_0:= d-q_1-q_2-q_3 \]
and consider the weighted projective space $\PP(Q)=\PP(q_0,q_1,q_2,q_3)$ with homogeneous coordinates $(w:x:y:z)$ (cf.\ \cite{Dolgachev82}). In this weighted projective space, we consider the quasismooth (i.e.\ the affine cone is smooth outside the vertex) hypersurface
\[ Z:=\{(w:x:y:z) \in \PP(q_0,q_1,q_2,q_3) \, | \, F(w,x,y,z)=0 \}, \]
where
\[F(w,x,y,z) = f(x,y,z) + w^{d/q_0}\]
in the case of the quadrilateral singularities and one of
\[
F(w,x,y,z) = \left\{ \begin{array}{l}
                        f(x,y,z) +zw^{(d-q_3)/q_0} \\
                        f(x,y,z) +yw^{(d-q_2)/q_0} \\
                        f(x,y,z) +xw^{(d-q_1)/q_0} \\
                        f(x,y,z) +zw^{(d-q_3)/q_0} + yw^{(d-q_2)/q_0} \\
                        f(x,y,z) +yw^{(d-q_2)/q_0} + xw^{(d-q_1)/q_0}
                     \end{array} \right.
\]
in the case of the 14 exceptional bimodal singularities. See Table~\ref{TabWPS} for the actual choice of deformation and compactification.

By \cite[3.3.4 Theorem]{Dolgachev82}, $Z$ is a simply connected projective surface with trivial dualizing sheaf $\omega_Z= \calO_Z$. Let $c:=c_f$. If the canonical system of weights is reduced, we set $Y:=Z$. Otherwise, we consider an action of the cyclic group $\ZZ_c=\ZZ/c\ZZ$ on $\PP(q_0,q_1,q_2,q_3)$, where a generator $\zeta \in \ZZ_c$ acts as follows
\[ (w:x:y:z) \mapsto (\zeta^{m_0} w: \zeta^{m_1} x: \zeta^{m_2}y: \zeta^{m_3}z)\]
and the corresponding quadruples $(m_0,m_1,m_2,m_3) \in \ZZ^4$ are indicated in Table~\ref{TabWPS}. This action leaves the surface $Z$ invariant. In these cases let $Y:=Z/\ZZ_c$ be the quotient variety.

\begin{proposition} \label{PropK3}
The variety $Y$ is a simply-connected projective surface with the dualizing sheaf $\omega_Y=\calO_Y$.
\end{proposition}

\begin{proof} Since the surface $Z$ is simply connected, it is clear that the surface $Y$ is still simply connected. Since $\omega_Z=\calO_Z$, the space of holomorphic 2-forms on $Z$  is generated by the holomorphic 2-form
\[ \omega_0 := \frac{q_0wdxdydz -q_1xdwdydz +q_2ydwdxdz -q_3zdwdxdy}{dF} \]
(cf.\ \cite{Saito87}). It is easy to see that this 2-form is invariant under the action of $\ZZ_c$.
\end{proof}

The singularities of $Y$ are cyclic quotient singularities. Let $\pi\colon X \to Y$ be a minimal resolution of its singularities.
By Proposition~\ref{PropK3}, $X$ is a smooth K3 surface. We summarise the relation between the three surfaces:
\[ \xymatrix{
X \phantom{XXXX} \ar@<-1.8em>[d]_(0.45){\text{resolution}}^(0.45)\pi
             & {}\save[]+<1cm,0cm> *\txt{smooth K3 surface} \restore \\
Y = Z/\ZZ_c  \\
Z = V(F)         \ar@<1.8em>[u]^(0.45){\text{covering}}
             & {}\save[]+<3cm,0cm> *\txt{hypersurface in weighted projective space; \\
                                         compactification of Berglund-H\"ubsch dual \\
                                         of a bimodal singularity~~~~~~~~~~~~~~~~~~~~~~~~~} \restore
} \]

\begin{table}[h]
\centering
\begin{tabular}{cllcc}
\toprule
Dual & \multicolumn{1}{c}{$F(w,x,y,z)$}  & $\PP(q_0,q_1,q_2,q_3)$  & $c$ & $(m_0,m_1,m_2,m_3)$ \\
\midrule
$Z_{13}$ & $x^6+xy^3+z^2+w^{18}$ & $\PP(1,3,5,9)$  & $2$ & $(0,1,-1,0)$  \\
$Z_{1,0}$ & $x^5y + xy^3 +z^2+w^{14}$ & $\PP(1,2,4,7)$  & $2$ &$(0,1,-1,0)$ \\
$Z_{17}$ & $x^4z + xy^3 +z^2+w^{12}$ &  $\PP(2,3,7,12)$ & $1$ &  \\
$W_{1,0}$ & $x^6+y^2z+z^2+w^{12}$ & $\PP(1,2,3,6)$  & $2$ & $(0,1,-1,0)$  \\
$W_{17}$  & $x^5y+y^2z+ z^2+w^{10}$ & $\PP(2,3,5,10)$ & $1$ & \\
$U_{1,0}$  & $x^3y+y^2z+z^3+w^9$ & $\PP(1,2,3,3)$ & $2$ &$(0,1,-1,0)$ \\
\midrule[0.1pt]
$Q_{12}$ & $x^5 + y^3 + xz^2+zw^9$ & $\PP(1,3,5,6)$  & $2$ & $(1,0,0,-1)$\\
$Z_{1,0}$ &  $x^7 + xy^3+z^2+yw^{10}$  & $\PP(1,2,4,7)$ & $3$ &$(1,0,-1,0)$ \\
$E_{20}$ & $x^{11} +y^3 +z^2+xw^{12}$ & $\PP(5,6,22,33)$ & $1$ & \\
\midrule[0.1pt]
$Q_{2,0}$ &  $x^4y + y^3 + xz^2+zw^7$ & $\PP(1,2,4,5)$ & $2$ & $(0,1,-1,0)$  \\
$Z_{18}$  & $x^6y + xy^3 + z^2+yw^8+xw^{10}$ & $\PP(3,4,10,17)$ & $1$ & \\
$E_{25}$  & $x^9y + y^3 + z^2+xw^{10}$ & $\PP(5,4,18,27)$ & $1$ &\\
\midrule[0.1pt]
$Q_{16}$ & $x^4z + y^3 + xz^2+zw^6+yw^7$ & $\PP(2,3,7,9)$ & $1$ & \\
$Z_{2,0}$  & $x^5z+xy^3+z^2+yw^7$  & $\PP(1,1,3,5)$   & $3$ & $(1,-1,-1,1)$ \\
$E_{30}$  & $x^8z+y^3+z^2+xw^9$ & $\PP(5,3,16,24)$ & $1$  & \\
\midrule[0.1pt]
$S_{1,0}$  & $x^5+xz^2+y^2z+zw^6+yw^7$ & $\PP(1,2,3,4)$   & $2$ & $(0,1,0,-1)$ \\
$W_{18}$  & $x^7+ y^2z+ z^2+xw^8$ & $\PP(3,4,7,14)$  & $1$ & \\
\midrule[0.1pt]
$S_{16}$  & $x^4y+xz^2+y^2z+zw^5+yw^6$ & $\PP(2,3,5,7)$ & $1$ & \\
$X_{2,0}$ & $x^6y+y^2z+z^2+xw^7$ & $\PP(1,1,2,4)$  & $3$ & $(1,-1,0,0)$ \\
\midrule[0.1pt]
$U_{16}$  & $x^5+y^2z+yz^2+xw^6$ & $\PP(2,3,5,5)$  & $1$ & \\
\bottomrule
\end{tabular}
\caption{Compactifications in weighted projective spaces \protect \footnotemark
} \label{TabWPS}
\end{table}

\section{Configuration of rational curves on $X$}
We want to study configurations of rational curves on $X$. Start by considering the curves
\[ \begin{array}{l @{\quad\text{and}\quad} l @{\quad\text{in }} l}
 C_\infty := (\{ w=0 \}\cap Z)/\ZZ_c  & C_0 := (\{ x=0 \}\cap Z)/\ZZ_c & Y, \\
 E_\infty := \pi^{-1}(C_\infty)         & E_0 := \pi^{-1}(C_0)           & X.
\end{array} \]

\footnotetext{The functions $F$ in this table differ from the published version for $Z_{18}$, $Q_{16}$, $S_{1,0}$ and $S_{16}$.}

\begin{proposition} \label{PropC0}
The curves $C_0$ and $C_\infty$ are rational curves on $Y$.
\end{proposition}

\begin{proof} The curves $\{ w=0 \} \cap Z$ and $\{ x=0 \} \cap Z$ are quasismooth weighted complete intersections in $\PP(Q)$ of multidegree $(d,q_0)$ and $(d,q_1)$ respectively. According to \cite[3.4.4 Corollary]{Dolgachev82}, their genus is equal to zero except in the cases $Z_{2,0}$ and $X_{2,0}$ where it is equal to one. If the genus is already zero in $Z$, then also for the image curve in $Y$. For $Z_{2,0}$ and $X_{2,0}$, the form
\[ \omega_1 := \frac{q_1xdydz-q_2ydxdz+q_3zdxdy}{df} \]
is a holomorphic 1-form on $\{ w=0 \} \cap Z$ which generates the space of holomorphic 1-forms on this curve. However, it is not invariant with respect to the action of the group $\ZZ_c$. A similar argument holds for the curve $\{ x=0 \} \cap Z$.
\end{proof}

The surface $Y$ has three cyclic quotient singularities of type $(\alpha_i, \alpha_i-1)$ ($i=1,2,3$) along the curve  $C_\infty$. The curve $C_0$ intersects the curve $C_\infty$ in some of these singularities. In order to compute how the curve $E_0$ meets the exceptional divisor of the resolution $\pi\colon X \to Y$, we study the local setting around a cyclic quotient singularity.

\bigskip

\noindent
\textbf{Local setting:} We first consider $\CC^2$ with the coordinates $x,y$ and an action of 
the cyclic group $\ZZ_k$ by $(x,y) \mapsto (\zeta x, \zeta^{-1} y)$ where $\zeta$ is a generator
of $\ZZ_k$. The quotient $\CC^2/\ZZ_k$ defines a cyclic quotient singularity of type $(k,k-1)$. 
It is well known that its resolution is obtained as follows:
The polynomials $x^k$, $y^k$, $xy$ are invariant under $\ZZ_k$. The map
\[ \psi \colon \CC^2 \to \CC^3, \quad (x,y) \mapsto (X,Y,Z)=(x^k,y^k,xy) \]
factors through $\CC^2/\ZZ_k$ and the image of the induced map is the hypersurface
\[ \{(X,Y,Z) \in \CC^3 \, | \, XY=Z^k \}. \]

The resolution $M \to \CC^2/\ZZ_k$ is obtained by glueing $k$ copies of $\CC^2$
(with coordinates $(u_i,v_i)$, $i=1, \ldots , k$) by the maps
\[ \phi_i \colon \CC^2 \setminus \{ v_i=0 \} \to \CC^2 \setminus \{ v_{i+1} = 0\}, \quad
   (u_i,v_i) \mapsto \left(\frac{1}{v_i}, u_iv_i^2 \right)=(u_{i+1}, v_{i+1}). \]

Considering the singularity as a hypersurface, the resolution is given by the mapping
$\pi_0\colon M \to \CC^3$ in the coordinates $(u_i,v_i)$ with
\[ (u_i,v_i) \mapsto (X,Y,Z)=(u_i^iv_i^{i-1}, u_i^{k-i}v_i^{k+1-i},u_iv_i). \]

The exceptional divisor is
\[ E = \bigcup_{i=1}^{k-1} E_i, \quad E_i = \{ u_i=v_{i+1} = 0 \}, \quad i=1, \ldots, k-1.\]
We have $E_i \cap E_{i+1} \neq \varnothing$ for $i=1, \ldots, k-1$ and $E_i \cap E_j = \varnothing$
otherwise. The dual graph corresponding to the components $E_i$ is a graph of type $A_{k-1}$.
Note that the proper preimage
of the curve $y=0$ under the resolution $M \to \CC^2/\ZZ_k$ intersects (transversally) the component $E_1$ of the exceptional divisor.

\begin{lemma} \label{LemC2}
Let $0<m <k$ be an integer. In $\CC^2$ with coordinates $x,y$ consider the curve $x^m+y^{k-m}=0$. Then the proper preimage of this curve under the resolution $M \to \CC^2/\ZZ_k$  intersects (transversally) the component $E_{k-m}$ of the exceptional divisor.
\end{lemma}

\begin{proof} Under the map $\psi$, the curve $x^m+y^{k-m}=0$ is mapped to the curve $Z^m+Y=0$. In the coordinates $(u_{k-m},v_{k-m})$ the preimage of this curve looks as follows:
\[ u_{k-m}^m v_{k-m}^m + u_{k-m}^mv_{k-m}^{m+1} = u_{k-m}^mv_{k-m}^m(1+v_{k-m}). \qedhere \]
\end{proof}

\begin{lemma} \label{LemC2:2}
In $\CC^2$ with coordinates $x,y$ consider the curve $x^2+y^{2k-2}=0$. Then the proper preimage of this curve under the resolution $M \to \CC^2/\ZZ_k$ has two components which intersect (transversally) the component $E_{k-1}$ of the exceptional divisor in two distinct points.
\end{lemma}

\begin{proof} Under the map $\psi$, the curve $x^2+y^{2k-2}=0$ is mapped to the curve $Z^2+Y^2=0$. In the coordinates $(u_{k-1},v_{k-1})$ the preimage of this curve looks as follows:
\[ u_{k-1}^2 v_{k-1}^2 + u_{k-1}^2v_{k-1}^4 = u_{k-1}^2v_{k-1}^2(1+v_{k-1}^2). \qedhere\]
\end{proof}

\noindent
\textbf{Application:} 
We use these lemmas to compute the configurations of smooth rational curves on $X$.
A smooth rational curve on a K3 surface has self-intersection number $-2$ by the adjunction formula.
For the 6 quadrilateral singularities, all the singularities of $Y$ lie on the curve $C_\infty$.
For the 14 exceptional bimodal singularities, the surface $Y$ has an additional singularity $P_0=(1:0:0:0)$. This is a cyclic quotient singularity of type $(a,a-1)$ where $a$ is defined in Section~\ref{Sect:bi}. It also lies on the curve $C_0$.
In the case $a=5$, Lemma~\ref{LemC2} implies that the curve $E_0$ intersects one of the inner components of the exceptional divisor corresponding to this singularity whose dual graph is of type $A_4$.
It turns out that the configurations of rational curves can be described with the help of Table~\ref{TabCD} in a similar way as the Coxeter-Dynkin diagrams:

\begin{proposition} \label{prop:config}
Let $f(x,y,z)$ be one of the invertible polynomials of Table~\ref{TabBi} with invariants 
$(\alpha_1, \beta_1), (\alpha_2, \beta_2), (\alpha_3, \beta_3)$ and let $Y$ be the surface constructed above.
Then the total transform of the curve $C_\infty$ under the resolution $\pi\colon X\to Y$ is a tree of smooth
rational curves with the proper transform $E_\infty$ as central curve and three branches of lengths
$\alpha_1, \alpha_2, \alpha_3$.
\begin{itemize}
\item[{\rm (i)}] If $f^T$ defines a singularity of Kodaira type ${\rm I}_0^\ast$ with $r=1$, the curve $E_0$ has two connected components $E_0'$ and $E_0''$. These are smooth rational curves which intersect the outermost curve of the third branch and no other component of the exceptional divisor.
\item[{\rm (ii)}] Otherwise, the curve $E_0$ is smooth and rational (in particular, irreducible).
If $\beta_i=\alpha_i-1$, then the curve $E_0$ does not intersect any curve of the $i$-th branch. Otherwise, the curve $E_0$ intersects the $\alpha_i - \beta_i +1$-th outermost curve of the $i$-th branch. If $P_0$ lies on $Y$, then $E_0$ also intersects one component of the exceptional divisor of the resolution of this singularity.
\end{itemize}
\end{proposition}

\begin{proof}
This is proved case by case using Lemma~\ref{LemC2} and Lemma~\ref{LemC2:2}.
We give some examples of this calculation.

\Example{$Z_{1,0}$}{Here, $F(w,x,y,z) = x^5y + xy^3 +z^2+w^{14}$ and
\[ Z:=\{(w:x:y:z) \in \PP(1,2,4,7) \, | \, F(w,x,y,z)=0 \} . \]
We first consider the chart
$U_1:= \{(w:x:y:z) \in \PP(1,2,4,7) \, | \, x=1 \}$.
Then $U_1= \CC^3/\ZZ_2$ where $\ZZ_2$ acts on $\CC^3$ by
$(w,y,z) \mapsto (-w, y, -z)$.
This action has 3 fixed points on
$Z_1:=\{(w,y,z) \in \CC^3 \, | \, F(w,1,y,z)=0 \}$,
namely
$P_1 = (0, \sqrt{-1},0)$, $P_1'=(0,-\sqrt{-1}, 0)$, and $P_2= (0,0,0)$.

Moreover, let
$U_2:= \{(w:x:y:z) \in \PP(1,2,4,7) \, | \, y=1 \}$.
Then $U_2= \CC^3/\ZZ_4$ where a generator $\zeta \in \ZZ_4$ acts on $\CC^3$ by
$(w,x,z) \mapsto (\zeta w, \zeta^2 x, \zeta^7 z)$.
The only fixed point on
$Z_2:=\{(w,x,z) \in \CC^3 \, | \, F(w,x,1,z)=0 \}$
is $P_3=(0,0,0)$. The surface
$Z_2=\{x^5 +x + z^2 + w^{14} = 0 \}$
is regular in $x$ and the $\ZZ_4$-action on the coordinates $(w,z)$ is given by $(w,z) \mapsto (\zeta w, \zeta^{-1} z)$. Therefore the surface $Z \cap U_2$ has an $A_3$ singularity in $P_3$.

We consider the action of the cyclic group $\ZZ_2$ on $\PP(1,2,4,7)$ given by
\[ (w:x:y:z) \mapsto (w: -x: -y: z).\]
Under this action, the two points $P_1$ and $P_1'$ are identified, $P_2$ gets a cyclic quotient singularity of type $(4,3)$, and $P_3$ becomes a cyclic quotient singularity of type $(8,7)$. The curve $\{ x = 0 \}$ only meets the point $P_3$. The singularity $P_3$ of $Y=Z/\ZZ_2$ is $Z_2/\ZZ_8$.
By Lemma~\ref{LemC2:2},  the proper preimage of the curve $C_0$ under the resolution $\pi\colon X \to Y$  consists of two components $E_0'$ and $E_0''$ which intersect (transversally) the component $E_7$ of the exceptional divisor $\pi^{-1}(P_3)$. Therefore we have the configuration depicted in Figure~\ref{FigZ10}.

\begin{figure}
\centering
\includegraph{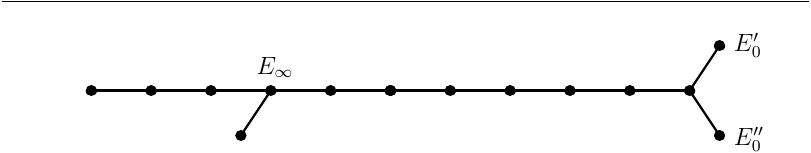}
\caption{  \label{FigZ10}
The configuration of rational curves in the case $Z_{1,0}$.}
\end{figure}

} 

\Example{$S_{16}$}{Here $F(w,x,y,z)=x^4y+xz^2+y^2z+zw^5+yw^6$ and 
\[ Y = Z :=\{(w:x:y:z) \in \PP(2,3,5,7) \, | \, F(w,x,y,z)=0 \} . \]
The surface $Y$ is quasismooth; note that this would be wrong without the term $yw^6$ because the
affine cone $x^4y+xz^2+y^2z+zw^5=0$ has  singularities along the curve $x=z=y^2+w^5=0$.\footnote{We thank Kazushi Ueda for this observation.}
This surface has 4 singularities, namely $P_0=(1:0:0:0)$ of type $A_1$, $P_1=(0:1:0:0)$ of type $A_2$, $P_2=(0:0:1:0)$ of type $A_4$, and finally $P_3=(0:0:0:1)$ of type $A_6$. The curve $\{x = 0\}$ goes through the points $P_0$, $P_2$, and $P_3$. One can easily see that it intersects the curve $\{ w=0\}$ transversally in the point $P_2$. To compute the intersection behaviour with the curve $\{w=0\}$ at the point $P_3$, consider the chart $U_3:=\{(w:x:y:z) \in \PP(2,3,5,7) \, | \, z=1 \}$. In this chart, $Y$ is given by the equation $x^4y+x+y^2+w^5+yw^6=0$. By an analytic change of the coordinates $(y,w)$ fixing the point $(y,w)=(0,0)$, one can get rid of the extra term $yw^6$. Then Lemma~\ref{LemC2} implies that the curve $E_0$ intersects the component $E_5$ of the exceptional divisor of the $A_6$ singularity $P_3$. Therefore we obtain the configuration depicted in Figure~\ref{FigS16}.

\begin{figure}
\centering
\includegraph{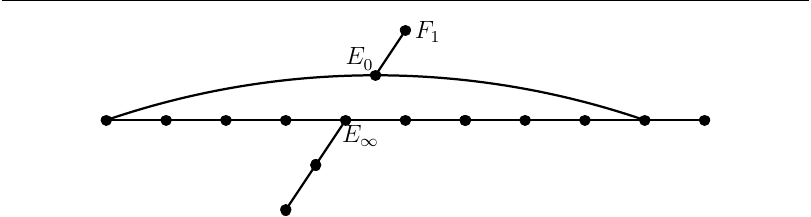}
\caption{  \label{FigS16}
The configuration of rational curves in the case $S_{16}$.}
\end{figure}

} 

\Example{$E_{20}$}{In this case $F(w,x,y,z) = x^{11} +y^3 +z^2+xw^{12}$ and
\[ Y = Z :=\{(w:x:y:z) \in \PP(5,6,22,33) \, | \, F(w,x,y,z)=0 \} . \]
We first consider the chart
$U_1:= \{(w:x:y:z) \in \PP(5,6,22,33) \, | \, x=1 \}$.
Then $U_1= \CC^3/\ZZ_6$  where a generator $\zeta \in \ZZ_6$ acts on $\CC^3$ by
$(w,y,z) \mapsto (\zeta^5w, \zeta^{22}y, \zeta^{66}z)$. The singularity $(0:1:0:0) \in U_1$ of the weighted projective space $\PP(5,6,22,33)$ does not lie on $Z$, but the invariant surface $Z_1:=\{Ê(w,y,z) \in \CC^3 \, | \, F(w,1,y,z)=0 \}$ has two points with non-trivial isotropy group, namely $P_1=(0,-1,0)$ with isotropy group of order 2 and $P_2=(0,0,\sqrt{-1})$ with isotropy group of order 3. They yield singularities of type $A_1$ and $A_2$ respectively. A similar reasoning for the chart
$U_2:= \{(w:x:y:z) \in \PP(5,6,22,33) \, | \, y=1 \}$ shows that the surface $Z$ has a third singularity $P_3=(0:0:1:\sqrt{-1})$ of type $A_{10}$. In this chart, $Z$ is given by
$Z_2:=\{Ê(w,x,z) \in \CC^3 \, | \, x^{11} +1 +z^2+xw^{12}=0 \}$. In local coordinates $(\xi_0, \xi_1, \xi_3)=(w, x, z-\sqrt{-1})$ around $P_3$, where $P_3$ becomes the origin, the equation of $Z_2$ is given by $\xi_1^{11} + \xi_3^2 + 2 \sqrt{-1} \xi_3 + \xi_1\xi_0^{12}=0$. This shows that the curve $\{ x=0 \}$ intersects the curve $\{ w=0 \}$ in $P_3$ transversally. Therefore the proper preimages of these curves under the resolution $\pi\colon X \to Y$ intersect the first and the last component of the exceptional divisor $\pi^{-1}(P_3)$ respectively.

Now the surface $Y$ has an additional singularity $P_0=(1:0:0:0)$. Consider the corresponding chart
$U_0:= \{(w:x:y:z) \in \PP(5,6,22,33) \, | \, w=1 \}$. Then $U_0= \CC^3/\ZZ_5$ where a generator $\zeta \in \ZZ_5$ acts on $\CC^3$ by $(x,y,z) \mapsto (\zeta^6 x, \zeta^{22} y, \zeta^{33} z)$. Therefore $P_0$ is an $A_4$ singularity. The curve $\{ x=0 \}$ in this chart is given by $y^3+z^2$. It follows from Lemma~\ref{LemC2} that the proper preimage of this curve under the resolution $\pi\colon X \to Y$ intersects the component $E_2$ of the exceptional divisor of $\pi^{-1}(P_0)$. Therefore we obtain the configuration depicted in Figure~\ref{FigE20}. \qedhere

\begin{figure}
\centering
\includegraph{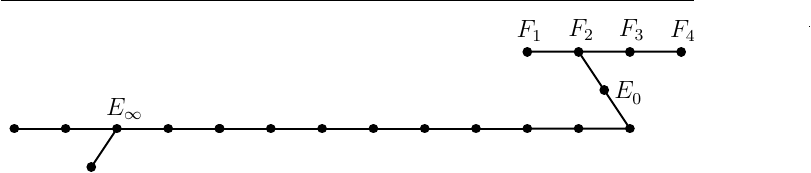}
\caption{ \label{FigE20}
The configuration of rational curves in the case $E_{20}$.}
\end{figure}

} 
\end{proof}

\section{Categories and Coxeter-Dynkin diagrams}
We have seen that the dual graphs of the curve configurations which we have constructed in the previous section are very similar to parts of the corresponding Coxeter-Dynkin diagrams of the bimodal singularities. We now want to realize the precise diagrams as
Coxeter-Dynkin diagrams corresponding to certain sets of generators in triangulated categories associated to the above curve configurations. Let us note right away that the construction is geometric: we are providing a collection of sheaves on $X$. In the end, we will come up with a category whose associated lattice from K-theory coincides with the Milnor lattice of the corresponding singularity.

All our categories will be built in the following way: Starting with a K3 surface
$X$ and a configuration of smooth rational $-2$-curves, we will consider the 
smallest triangulated subcategory $\calT$ of the bounded derived category $\DDD^b(X)$
(of coherent sheaves) which is generated by the structure sheaf $\calO_X$ and line 
bundles supported on $-2$-curves. In certain cases, we have to apply a base change by
way of a spherical twist.

Regarding the Ext groups of those sheaves, the relevant facts are collected in the following statement,
where we make use of the complex $\Hom^\bullet(A,B)=\Hom(A,B)\oplus\Ext^1(A,B)[-1]\oplus\Ext^2(A,B)[-2]$
for sheaves $A,B$ on $X$ (this is a complex with zero differentials, so can be seen as a graded vector space).

\begin{lemma} \label{LemSphericals}
Let $X$ be a K3 surface and $C,D\subset X$ be two smooth rational $-2$-curves. Then 
$\Hom^\bullet(\calO_X,\calO_C)=\CC$, $\Hom^\bullet(\calO_X,\calO_C(-1))=0$,
$\Hom^\bullet(\calO_C,\calO_C(-1))=\CC^2[-2]$. Furthermore, if $C$ and $D$ intersect transversally then
$\Hom^\bullet(\calO_C(i),\calO_D(j))=\CC[-1]$ for any $i,j\in\ZZ$, whereas
$\Hom^\bullet(\calO_C(i),\calO_D(j))=0$ if $C$ and $D$ are disjoint.
\end{lemma}

Since the canonical bundle of $X$ is trivial, the Serre functor of $\DDD^b(X)$ is just the shift $[2]$, and the same is then true for $\calT$. Such a category is often called a `2-Calabi-Yau category'. 
This implies that the Grothendieck K-group $K(\calT)$, equipped with the negative Euler pairing
\[ -\chi([A],[B])=-\sum_{i \in \ZZ}(-1)^i\dim\Hom_{\calT}(A,B[i]),
\] 
is a lattice. 
What is more, $\calT$ will be generated by spherical objects, i.e.\ objects $S\in\calT$ with $\Hom^\bullet(S,S)=\CC\oplus\CC[-2]$. Such objects give rise to roots $[S] \in K(\calT)$. The structure sheaf $\calO_X$ is spherical --- this is just rephrasing the fact that $X$ is a K3 surface. It is well-known that a line bundle on a chain of $-2$-curves is spherical. And as is standard by now, a spherical object $S$ gives rise to an autoequivalence $\TTT_S$ of the category, the \emph{spherical twist} associated to $S$. Since $\calT$ is 2-Calabi-Yau, the autoequivalence $\TTT_S$ descends to the reflection of $(K(\calT), - \chi(-,-))$ induced by the root $[S]$. 

According to Proposition~\ref{prop:config}, the surface $X$ comes with a star-like configuration of $-2$-curves, given by $\pi^{-1}(C_\infty)$. This graph has three arms of lengths $\alpha_1$, $\alpha_2$, $\alpha_3$. We denote the corresponding curves by $\E{i}{j}$ where $i=1,2,3$ and $j=1,\dots,\alpha_i-1$, starting at the outer ends. The central vertex corresponds to the curve $\E{}{\infty}$, it meets the curves $\E{i}{\alpha_i-1}$. Furthermore, there is always the curve $\E{}{0}$, as the strict transform of $C_0$; in three cases it decomposes into two components $E_0'$ and $E_0''$. For the 14 exceptional singularities, there are additional $-2$-curves from resolving the cyclic quotient singularity $P_0$; we call them $F_\ell$.

The situation is simplest for the singularities dual to the bimodal singularities with $a=2$ except those of Kodaira type ${\rm I}_0^\ast$ with $r=1$ ($J_{3,0}$, $Z_{1,0}$, $Q_{2,0}$). In this case, the curve configuration consists of the central curve $\E{}{\infty}$, the three arms $\E{i}{j}$ and the additional curve $\E{}{0}$. In the case of the exceptional bimodal singularities, we have an additional curve $F_1$ coming from the singularity $P_0$. This will not be used. We define the category $\calT$ as the smallest triangulated subcategory of $\DDD^b(X)$ containing the following objects:

\begin{description*}
\item[Case $a=2$] \hfill ($W_{1,0}, S_{1,0}, U_{1,0}, E_{18}, Z_{17}, Q_{16}, W_{17}, S_{16}, U_{16}$)
\begin{align*}
 \calT = \big\langle
      &  \calO_{\E{1}{1}}(-1), \ldots, \calO_{\E{1}{\alpha_1-1}}(-1),
         \calO_{\E{2}{1}}(-1), \ldots, \calO_{\E{2}{\alpha_2-1}}(-1), \\
      &  \calO_{\E{3}{1}}(-1), \ldots, \calO_{\E{3}{\alpha_3-1}}(-1),
         \calO_{\E{}{\infty}}(-1), \calO_{\E{}{\infty}}, \calO_X, \calO_{\E{}{0}}
         \big\rangle
\end{align*}
\end{description*}

\begin{figure}
\centering
\includegraph{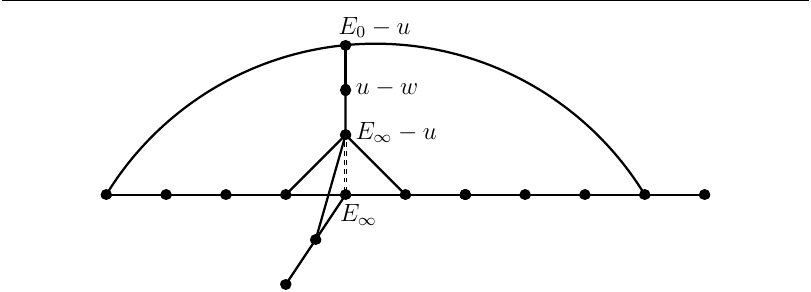}
\caption{Coxeter-Dynkin diagram for $S_{16}$.}  \label{FigS16full}
\end{figure}

The K-group of this category is the lattice spanned by the curves of $\pi^{-1}(C_\infty)$
and $E_0$, extended by a hyperbolic plane with a basis of isotropic elements $u$ and $w$;
we use the roots $\E{}{\infty}-u$ and $u-w$ as generators.
See Figure~\ref{FigS16full} for the singularity $S_{16}$. The lattice $K(\calT)$ can be
seen as a sublattice of the cohomology $H^*(X,\ZZ)$ equipped with the Mukai pairing. In
this picture, $-u$ is the class of a skyscraper sheaf (of length 1) on the curve 
$\E{}{\infty}$, i.e.\ an element of $H^4(X,\ZZ)$. The isotropic element $-w$ corresponds
to the ideal sheaf of this point, so that $\chern(\calO_X)=u-w$. The correspondence 
between sheaves and lattice elements is furnished by the Chern character
 $\chern\colon \calT\hookrightarrow\DDD^b(X)\to H^*(X,\ZZ)$.
Then $\chern(\calO_{\E{}{\infty}}(-1))=E_\infty$ in $H^2(X,\ZZ)$ and similar for the
other curves. Furthermore, $\chern(\calO_{\E{}{\infty}})=\E{}{\infty}-u$.
For details see \cite{EP}.

In the case $a=3$, we use in addition the curve $F_1$ of the exceptional divisor of
the $A_2$ singularity $P_0$, but not $F_2$:
\begin{description*}
\item[Case $a=3$] \hfill ($E_{19}, Z_{18}, Q_{17}, W_{18}, S_{17}$)
\begin{align*}
 \calT = \big\langle
       & \calO_{E^1_1}(-1), \ldots, \calO_{E^1_{\alpha_1-1}}(-1),
         \calO_{E^2_1}(-1), \ldots, \calO_{E^2_{\alpha_2-1}}(-1), \\
       & \calO_{E^3_1}(-1), \ldots, \calO_{E^3_{\alpha_3-1}}(-1),
         \calO_{E_\infty}(-1), \calO_{E_\infty}, \calO_X, \calO_{F_1}(-1), \calO_{E_0}
         \big\rangle
\end{align*}
\end{description*}

The remaining cases are the singularities dual to the bimodal singularities with $a=2$
of Kodaira type ${\rm I}_0^\ast$ with $r=1$ ($J_{3,0}$, $Z_{1,0}$, $Q_{2,0}$)  and the 
three singularities with $a=5$ of Kodaira type ${\rm II}^\ast$. In each case, there is
one curve inside the third branch of the curve configuration which we have to omit.

To this end, we will apply a suitable base change. Let us denote the superfluous
curve momentarily by $B$. For the sake of simplicity, we assume that $B$ is
incident to just two other smooth rational curves $A$ and $C$. The base change
we are after is $[C]\mapsto[B]+[C]$; note that this is the reflection along the root
$B$ applied to $C$. Omitting the curve $B$ leaves us with a chain one vertex shorter,
as desired:

\begin{center}
\includegraph{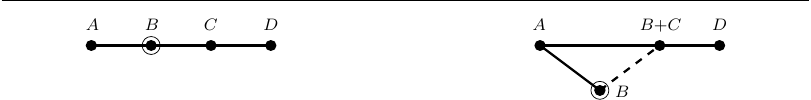}
\end{center}


On the categorical level, we use that the spherical twist $\TTT_{\calO_B(-1)}$ is a 
lift of the reflection, i.e.\ we use the sheaf $\TTT_{\calO_B(-1)}(\calO_C(-1))$ which
is defined by the short exact sequence
 $ 0 \to \calO_C(-1) \to \TTT_{\calO_B(-1)}(\calO_C(-1)) \to \calO_B(-1) \to 0 . $
This non-split extension is unique as a sheaf and a line bundle supported on $B\cup C$.
It follows immediately from this sequence and Lemma~\ref{LemSphericals} that the
$\Hom^\bullet$-groups are preserved. As a consequence of this, the intersection behaviour,
given by the negative of the Euler form on the category, is unchanged. 

We define the category $\calT$ as the smallest triangulated subcategory of $\DDD^b(X)$
containing the following objects:
\begin{description*}
\item[Case $a=2$] \hfill ($J_{3,0}, Z_{1,0}, Q_{2,0}$)
\begin{align*}
 \calT = \big\langle
       & \calO_{\E{1}{1}}(-1), \ldots, \calO_{\E{1}{\alpha_1-1}}(-1),
         \calO_{\E{2}{1}}(-1), \ldots, \calO_{\E{2}{\alpha_2-1}}(-1), \\
       & \TTT_{\calO_{\E{3}{1}}(-1)}(\calO_{\E{3}{2}}(-1)), 
         \calO_{\E{3}{3}}(-1), \ldots, \calO_{\E{3}{\alpha_3-1}}(-1), \\
%
%
       & \calO_{\E{}{\infty}}(-1), \calO_{\E{}{\infty}}, \calO_X, \calO_{\E{}{0}}
         \big\rangle
\end{align*}
\item[Case $a=5$] \hfill ($E_{20}, Z_{19}, Q_{18}$)
\begin{align*}
 \calT = \big\langle
       & \calO_{\E{1}{1}}(-1), \ldots, \calO_{\E{1}{\alpha_1-1}}(-1),
         \calO_{\E{2}{1}}(-1), \ldots, \calO_{\E{2}{\alpha_2-1}}(-1), \\
        & \TTT_{\calO_{\E{3}{1}}(-1)}(\calO_{\E{3}{2}}(-1)), 
         \calO_{E^3_3}(-1), \ldots , \calO_{E^3_{\alpha_3-1}}(-1), \\
%
       & \calO_{E_\infty}(-1), \calO_{E_\infty}, \calO_X,
         \calO_{F_1}, \calO_{F_2}(-1), \calO_{F_3}(-1), \calO_{F_4}(-1), \calO_{E_0}(-1)
         \big\rangle .
\end{align*}
\end{description*}

The Coxeter-Dynkin diagrams corresponding to these sets of generators for the 
singularities $Z_{1,0}$ and $E_{20}$ are depicted in Figure~\ref{FigZ10full} and 
Figure~\ref{FigE20full} respectively.

\begin{figure}
\centering
\includegraph{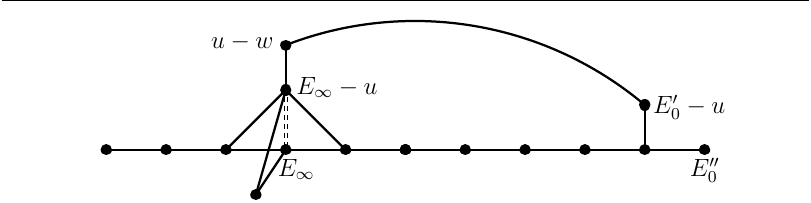}
\caption{  \label{FigZ10full}
Coxeter-Dynkin diagram for $Z_{1,0}$.}
\end{figure}

\begin{figure}
\centering
\includegraph{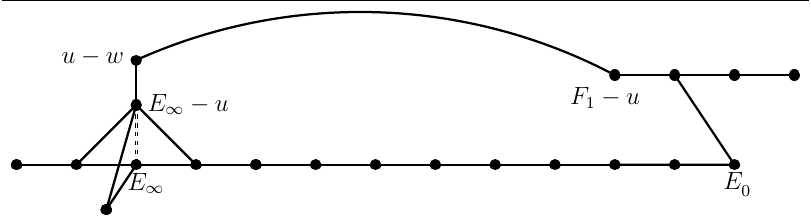}
\caption{Coxeter-Dynkin diagram for $E_{20}$.}  \label{FigE20full}
\end{figure}

Summarising, we obtain the following theorem.

\begin{theorem}
Let $\calT$ be one of the triangulated categories associated above with a bimodal
singularity. Then the lattice $K(\calT)$, equipped with the negative Euler pairing,
is isomorphic to the Milnor lattice of the singularity  and the Coxeter-Dynkin 
diagram corresponding to the above system of generators of $\calT$ coincides with
the Coxeter-Dynkin diagram corresponding to a distinguished basis of vanishing 
cycles of the singularity.
\end{theorem}

\section{Coxeter elements}
Let $\calT$ be one of the above categories. A spherical object $D$ in $\calT$ gives rise to a spherical twist whose action on $(K(\calT), - \chi(-,-))$ is just the reflection $s_{[D]}$ along the class $[D]\in K(\calT)$.

\begin{corollary} Let $\calT$ be one of the triangulated categories associated above with a bimodal singularity. The Coxeter element corresponding to the above system of generators of $\calT$ corresponds to the monodromy operator of the singularity.
\end{corollary}

\begin{remark}
Since the triangulated category $\calT$ is generated by 2-spherical objects,  there is a Coxeter functor, given by composing all the spherical twists of the spherical objects comprising the basis of $\calT$. This functor lifts the Coxeter element from an isometry of the lattice to an autoequivalence of $\calT$.
\end{remark}

If $\tau$ is the monodromy operator, then we consider the polynomial $\Delta(t)=\det(1-\tau^{-1}t)$ as  its characteristic polynomial, using a suitable normalization.

Let $f(x,y,z)$ be a non-degenerate invertible polynomial and let $(w_1,w_2,w_3;d')$ be the canonical weight system corresponding to $f(x,y,z)$.
The ring $R_f:=\CC[x,y,z]/(f)$ is a $\ZZ$-graded ring. Therefore, we can consider the decomposition of $R_f$
as a $\ZZ$-graded $\CC$-vector space:
\[
R_f:=\bigoplus_{k\in\ZZ_{\ge 0}} R_{f,k}, \quad R_{f,k}:=\left\{g\in R_f~\left|~
w_1x\frac{\partial g}{\partial x}+w_2y\frac{\partial g}{\partial y}
+w_3z\frac{\partial g}{\partial z}=k g\right.\right\}.
\]
The formal power series
\begin{equation}
p_f(t):=\sum_{k\ge 0} (\dim_\CC R_{f,k}) t^k
\end{equation}
is the {\em Poincar\'e series} of the $\ZZ$-graded coordinate ring $R_f$ with respect to the canonical system of weights $(w_1,w_2,w_3;d')$ attached to $f$. It is given by
\[ p_f(t) = \frac{(1-t^{d'})}{(1-t^{w_1})(1-t^{w_2})(1-t^{w_3})}. \]
Let $(\alpha_1, \alpha_2, \alpha_3)$ be the Dolgachev numbers of $f$ (see \cite{ET}). Consider the polynomial
\[ \Delta_0(t) = (1-t)^{-2} (1- t^{\alpha_1})(1- t^{\alpha_2})(1- t^{\alpha_3}). \]
The rational function
\[ \phi_f(t) := p_f(t) \Delta_0(t) \]
is called the {\em characteristic function of $f$}. From Table~10 and Table~11 of \cite{ET} we can derive the following theorem.

\begin{theorem} \label{theo:char}
Let $f(x,y,z)$ be a non-degenerate invertible polynomial and assume that the canonical system of weights attached to $f^T$ is reduced. Then $\phi_f(t)$ is the characteristic polynomial  of the monodromy operator of $f^T$.
\end{theorem}

If $f^T$ is the invertible polynomial in Table~\ref{TabBi} corresponding to one
of the 14 exceptional bimodal singularities, then its canonical system of weights
is reduced. 
Therefore we can apply Theorem~\ref{theo:char} in these cases. Note that in these
cases $\Delta_0(t)$ is the characteristic polynomial of the Coxeter element 
corresponding to the subset of generators of $\calT$ with support on the preimage
of $C_\infty$ under the resolution $\pi\colon X \to Y$ with the Coxeter-Dynkin 
diagram given by Figure~\ref{FigTpqr}. Then we get the following corollary of 
Theorem~\ref{theo:char}:

\begin{corollary} \label{cor:Poinc}
Let $f(x,y,z)$ be an invertible polynomial which is the Berglund-H\"ubsch
transpose of an invertible polynomial with a reduced canonical system of weights
defining an exceptional  bimodal singularity. Then
\[ p_f(t) = \frac{\Delta(t)}{\Delta_0(t)} \]
where $\Delta(t)$ is the characteristic polynomial of the Coxeter element
corresponding to the above system of generators of $\calT$.
\end{corollary}

A similar result holds for Fuchsian singularities \cite{EP,EP2}. There we gave
a geometric proof of this fact. It is an open problem to derive a similar proof
for Corollary~\ref{cor:Poinc}.

\begin{remark} Since the canonical systems of weights of $Z_{17}$ and $W_{17}$
are reduced, we can apply \cite[Theorem~22]{ET} and obtain that $\phi_f(t)$ is
the characteristic polynomial of an operator $\tau$ such that $\tau^2$ is the 
Coxeter element corresponding to the above system of generators of $\calT$.  
It can be checked that a similar result holds for $U_{1,0}$. For the remaining
quadrilateral singularities, there is no such relation between $\phi_f(t)$ and
the Coxeter element.
\end{remark}

\noindent Leibniz Universit\"{a}t Hannover, Institut f\"{u}r Algebraische Geometrie,\\
Postfach 6009, D-30060 Hannover, Germany \\
E-mail: ebeling@math.uni-hannover.de\\
E-mail: ploog@math.uni-hannover.de

\end{document}